\documentclass{article}
\usepackage[all]{xy}
\usepackage{amsmath}
\usepackage{amssymb}
\usepackage{amsfonts}
\usepackage{amsthm}
\usepackage{array}
\usepackage{graphicx}
\usepackage{float}
\usepackage{color}
\usepackage{fullpage}
\newcommand{\Char}{\mathrm{char}\:}
\newcommand{\bbf}{\mathbb{F}}
\newcommand{\catmod}{\textbf{mod}}
\newtheorem{lemma}{Lemma}[section]
\newtheorem{thm}[lemma]{Theorem}
\newtheorem*{thm*}{Theorem}
\newtheorem{propn}[lemma]{Proposition}
\newtheorem{cor}[lemma]{Corollary}
\theoremstyle{definition}
\newtheorem{definition}[lemma]{Definition}
\newtheorem*{definition*}{Definition}
\newtheorem{example}[lemma]{Example}
\theoremstyle{remark}
\newtheorem*{rem*}{Remark}
\newtheorem{rem}[lemma]{Remark}
\begin{document}
\title{The blocks of the Brauer algebra in \mbox{characteristic $p$}}
\author{Oliver King}
\date{}
\maketitle
\begin{abstract}
Brauer algebras form a tower of cellular algebras. There is a well-defined notion of limiting blocks for these algebras. In this paper we give a complete description of these limiting blocks over any field of positive characteristic. We also prove the existence of a class of homomorphisms between cell modules.
\end{abstract}

\section{Introduction}
Classical Schur-Weyl duality relates the representation theory of the symmetric group $S_n$ and general linear group $GL_m$, via their action on the tensor space $(\mathbb{C}^m)^{\otimes n}$. Given this setup it is natural to ask if it is possible to find other algebras that centralise each others' actions, in particular if we replace the general linear group by the orthogonal group $O_m$ or the symplectic group $Sp_{m}$ (in the case $m$ is even). The Brauer algebra $B_n(\delta)$ was introduced in \cite{brauer} to provide this corresponding dual for suitable integral values of $\delta$.

It is possible however, to define the Brauer algebra $B_n^R(\delta)$ over an arbitrary ring $R$, for any $n\in\mathbb{N}$ and $\delta\in R$. Then, rather than examining it as a way of understanding its corresponding centraliser algebra in the context of Schur-Weyl duality, we study the representation theory of $B_n^R(\delta)$ in its own right.

Graham and Lehrer \cite{grahamlehrer} showed that the Brauer algebra $B_n^\bbf(\delta)$ over a field $\bbf$ is a cellular algebra, with cell modules indexed by partitions of $n,n-2,
 n-4,\dots,0$ or $1$ (depending on the parity of $n$). If \mbox{$\Char \bbf=0$} then these partitions also label a complete set of non-isomorphic simple modules, given by the heads of the corresponding cell modules. If $\Char \bbf=p>0$ then the simple modules are indexed by the subset of $p$-regular partitions. The question of computing decomposition matrices for these cell modules over any field is then raised, a problem which remains open in positive characteristic.

However, much work has been done in computing the block structure of these algebras, a portion of which is summarised below.

Wenzl \cite{wenzl} proved that over $\mathbb{C}$, the Brauer algebra is semisimple for all non-integer values of $\delta$. Motivated by this, Rui \cite{rui} provided a necessary and sufficient condition for semisimplicity, valid over an arbitrary field.

We say that two partitions are in the same block for $B_n^\bbf(\delta)$ if they label cell modules in the same block. A necessary and sufficient condition for two partitions to be in the same block if $\Char \bbf=0$ was given in \cite{blocks}, using the theory of towers of recollement \cite{cmpx}. It was shown in \cite{geom} that this is equivalent to partitions being in the same orbit under some action of a Weyl group $W$ of type $D$. In the same paper, it was found that in the case \mbox{$\Char \bbf=p>2$}, the orbits of the corresponding affine Weyl group $W_p$ of type $D$ on the set of partitions correspond to unions of blocks of the Brauer algebra $B_n^\bbf(\delta)$.

In this paper we investigate the representations of $B_n^\bbf(\delta)$ when \mbox{$\Char\bbf=p>0$}. Following \cite{cmpx} we can embed the module categories $B_n^\bbf(\delta)$-\catmod\, in $B_{n+2}^\bbf(\delta)$-\catmod, and we will see that this leads to a well-defined limit of the blocks as $n\rightarrow\infty$. We will then use an adaptation of the abacus method of representing partitions \cite{jameskerber} and the geometric results from \cite{geom} to show that these limiting blocks correspond precisely to the orbits of the affine Weyl group of type $D$ on the set of partitions.

In Section 2 we provide the notation and definitions used subsequently, as well as a review of some basic results. Section 3 introduces a variation of the abacus and describes some movement of beads across runners.

Section 4 makes use of this to derive some results regarding the blocks of the Brauer algebra. In particular we have the main result of this paper \mbox{(Theorem \ref{thm:limorbs})}, which states that in the limiting case the blocks are equal to the orbits of the affine Weyl group. In Section 5 we prove the existence of a class of homomorphisms between cell modules.

\section{Preliminaries}

In this section we will set up the framework for what follows, briefly review the modular representation theory of the symmetric group and present some background on the representation theory of the Brauer algebra.\\
\\
Let $(K,R,k)$ be a $p$-modular system, i.e.
	\begin{itemize}
		\item $R$ is a discrete valuation ring
		\item $R$ has maximal ideal $\mathfrak{m}=(\pi)$
		\item $K=\mathrm{Frac}(R)$ is the field of fractions
		\item $k=R/\mathfrak{m}$ is the residue field of characteristic $p$.
	\end{itemize}
Now let $A$ be an $R$-algebra, free and of finite rank as an $R$-module. We can extend scalars to produce the $K$-algebra \mbox{$KA=K\otimes_RA$} and the $k$-algebra \mbox{$kA=k\otimes_RA$}. Given an $A$-module X, we can then also consider the $KA$-module $KX=K\otimes_RX$ and the $kA$-module $kX=k\otimes_RX$.

As an $R$-algebra, $A$ can be uniquely decomposed into a direct sum of subalgebras
	\[A=\bigoplus A_i\] 
where each $A_i$ is indecomposable as an algebra. We call the $A_i$ the \emph{blocks} of $A$. For any $A$-module $M$ there exists a similar decomposition
	\[M=\bigoplus M_i\]
where for each $i$,
	\[A_iM_i=M_i,~~~A_jM_i=0~(\forall j\neq i)\]
We say that the module $M_i$ \emph{lies in the block} $A_i$. Clearly, each simple $A$-module must lie in precisely one block.

In what follows, we will use the terms $K$-block and $k$-block to indicate that we are considering the blocks of the algebra over the fields $K$ and $k$ respectively.

\subsection*{Modular representation theory of the symmetric group}

A more detailed account of the results in this section can be found in \cite{jameskerber}.\\
\\
Let $\bbf$ be a field. We denote by $S_n$ the symmetric group on $n$ letters, and by $\bbf S_n$ the corresponding group algebra over the field $\bbf$. For each partition $\lambda=(\lambda_1,\dots,\lambda_l)$ of $n$ there is a corresponding $\bbf S_n$-module $S^\lambda_\bbf$, called a \emph{Specht module}. If $\bbf=K$, then these provide a complete set of non-isomorphic simple $KS_n$-modules. However if $\bbf=k$, then this may not be the case.\\ 
We say that a partition $\lambda=(\lambda_1,\dots,\lambda_l)$ is \emph{$p$-singular} if there exists $t$ such that
	\[\lambda_t=\lambda_{t+1}=\dots=\lambda_{t+p-1}>0\]
Partitions that are not $p$-singular we call \emph{$p$-regular}.\\
Following \cite{jameskerber} we see that for each $p$-regular partition $\lambda$, the Specht module $S^\lambda_k$ has a simple head $D^\lambda_k$, and that these form a complete set of non-isomorphic simple $kS_n$-modules.\\
\\
With this information, we now provide a description of the $k$-blocks of this algebra. To each partition $\lambda$ we may associate the Young diagram 
	\[[\lambda]=\{(x,y)~|~x,y\in\mathbb{Z},~1\leq x\leq l,~1\leq y \leq \lambda_x\}\]
An element $(x,y)$ of $[\lambda]$ is called a \emph{node}.\\
The \emph{hook} $h_{(x,y)}$ corresponding to the node (x,y) in the Young diagram is the subset
	\[ h_{(x,y)}=\{(i,j)\in[\lambda]~|~i\geq x,~j\geq y\} \]
consisting of $(x,y)$ and all the nodes either below or to the right of it. A $p$-hook is a hook containing $p$-nodes.

	\begin{figure}[H]
		\centering
		\includegraphics[scale=0.3]{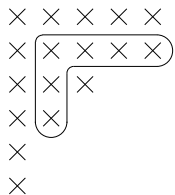}
		\caption{The Young diagram of $\lambda=(5^2,3,2,1^2)$ with the hook $h_{(2,2)}$ circled}
		\label{fig:younghook}
	\end{figure}

Each hook corresponds to a \emph{rim hook}, obtained by moving each node of the hook down and to the right so that it lies on the edge of the Young diagram. An example is given below.

	\begin{figure}[H]
		\centering
		\includegraphics[scale=0.3]{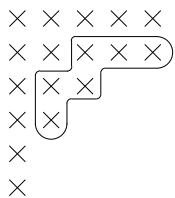}
		\caption{The rim hook corresponding to $h_{(2,2)}$. The node $(2,2)$ has moved to the rim.}
		\label{fig:rimhook}
	\end{figure}

Given a partition $\lambda$ we can successively remove rim $p$-hooks until we have reached a point that we can remove no more. What remains is called the $p$-core of $\lambda$, and the number of rim hooks removed to reach this is called the $p$-weight. It is shown in \cite[Chapter 2.7]{jameskerber} that the $p$-core is independent of the order in which we remove such hooks, and therefore both these notions are well-defined. We now state a useful theorem:

	\begin{thm}[Nakayama's Conjecture]
		Two partitions $\lambda$ and $\mu$ label Specht modules in the same $k$-block for the symmetric group algebra if and only if they have the same $p$-core and $p$-weight.
		\label{thm:nakayama}
	\end{thm}
	
A proof of this result can be found in \cite[Chapter 6]{jameskerber}.

\subsection*{The Brauer algebra}

A more detailed account of the results in this section can be found in \cite{blocks} and \cite{geom}.\\
\\
For a fixed $\delta\in R$ and $n\in\mathbb{N}$, the Brauer algebra $B_n^R(\delta)$ can be defined as the set of linear combinations of diagrams with $2n$ nodes, arranged in two rows of $n$, and $n$ arcs between them so that each node is joined to precisely one other. Multiplication of two diagrams is by concatenation in the following way: to obtain the result $x\cdot y$ given diagrams $x$ and $y$, place $x$ on top of $y$ and identify the bottom nodes of $x$ with those on the top of $y$. This new diagram may contain a number, $t$ say, of closed loops. These we remove and multiply the final result by $\delta^t$. An example is given in Figure \ref{fig:brauermult} below.

	\begin{figure}[H]
		\centering
		\includegraphics{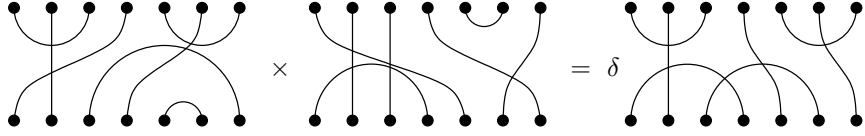}
		\caption{Multiplication of two diagrams in $B_7^R(\delta)$}
		\label{fig:brauermult}
	\end{figure}
	
We may then define the algebras 
	\[ B_n^K(\delta)=K\otimes_RB_n^R(\delta)\text{~~~~and~~~~}B_n^k(\bar\delta)=k\otimes_RB_n^R(\delta)\]
where $\bar\delta$ is the modular reduction of $\delta$ in $k$.\\
\\
Fix a field $\bbf$ (for our purposes this will be either $K$ or $k$) and assume that $\delta\neq0$ in $\bbf$. For each $n\geq2$ we have an idempotent $e_n\in B_n^\bbf(\delta)$ as illustrated in Figure \ref{fig:idem}.

	\begin{figure}[H]
		\centering
		\includegraphics{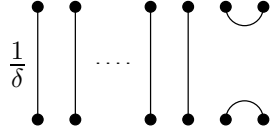}
		\caption{The idempotent $e_n$}
		\label{fig:idem}
	\end{figure}
	
Using these idempotents we can define algebra isomorphisms
	\begin{equation}
		\Phi_n:B_{n-2}^\bbf(\delta) \longrightarrow e_nB_n^\bbf(\delta)e_n \label{eq:phin}
	\end{equation}
taking a diagram in $B_{n-2}^\bbf(\delta)$ to the diagram in $B_n^\bbf(\delta)$ obtained by adding an extra northern and southern arc to the right hand end.
Using this and following \cite{green} we obtain an exact localisation functor
	\begin{eqnarray}
		F_n:B_n^\bbf(\delta)\text{-\catmod}	&\longrightarrow&	B_{n-2}^\bbf(\delta)\text{-\catmod}\label{eq:fn}\\
					 			 M 	&\longmapsto&    	e_nM\nonumber
	\end{eqnarray}
and a right exact globalisation functor
	\begin{eqnarray}
		G_n:B_n^\bbf(\delta)\text{-\catmod}	&\longrightarrow&	B_{n+2}^\bbf(\delta)\text{-\catmod}\label{eq:gn}\\
					  M 				&\longmapsto&	B_{n+2}^\bbf(\delta)e_{n+2}\otimes_{B_n^\bbf(\delta)}M\nonumber
	\end{eqnarray}
Since $F_{n+2}G_n(M)\cong M$ for all $M\in B_n^\bbf(\delta)$-\catmod, $G_n$ is a full embedding of categories. Also as
	\begin{equation} 
		B_n^\bbf(\delta)/B_n^\bbf(\delta)e_nB_n^\bbf(\delta) \cong \bbf S_n
		\label{eq:symmquotient}
	\end{equation}
the group algebra of the symmetric group on $n$ letters, we have from \eqref{eq:phin} and \cite{green} that the simple $B_n^\bbf(\delta)$-modules are indexed by the set
	\begin{equation}
		\Lambda_n=\Lambda^n\sqcup\Lambda_{n-2}=\Lambda^n\sqcup\Lambda^{n-2}\sqcup\dots\sqcup\Lambda^{0/1}
		\label{eq:indexset}
	\end{equation}
(depending on the parity on $n$), where $\Lambda^n$ is an indexing set for simple \mbox{$\bbf S_n$-modules}.\\
\\
Graham and Lehrer \cite{grahamlehrer} showed that $B_n^\bbf(\delta)$ is a cellular algebra, with cell modules $\Delta_n^\bbf(\lambda)$, indexed by partitions \mbox{$\lambda \vdash n,n-2,$} \mbox{$n-4,\dots,0/1$} (depending on the parity of $n$). When $\lambda\vdash n$, this is simply a lift of the Specht module $S^\lambda_\bbf$ to the Brauer algebra using \eqref{eq:symmquotient}. When $\lambda\vdash n-2t$ for some $t>0$, we obtain the cell module by
	\[ \Delta_n^\bbf(\lambda)\cong G_{n-2}G_{n-4}\dots G_{n-2t}\Delta_{n-2t}^\bbf(\lambda) \]
	
Over $K$, each of these modules has a simple head $L_n^K(\lambda)$, and these form a complete set of non-isomorphic simple $B_n^K(\delta)$-modules.

Over $k$, the heads $L_n^k(\lambda)$ of cell modules labelled by $p$-regular partitions provide a complete set of non-isomorphic simple $B_n^k(\delta)$-modules.
 \\ \\
We can also define modules $\Delta_n^R(\lambda)$ over $R$, and as shown in \cite{hartmannpaget} these will be cell modules for $B_n^R(\delta)$. Let a \emph{$(n,t)$-partial digram} be a row of $n$ nodes with $t$ edges, so that each node is connected to at most one other. We call those nodes not connected to another $\emph{free nodes}$.

	\begin{figure}[H]
	\centering
	\includegraphics{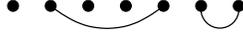}
	\caption{An example of a (7,2)-partial diagram}
	\end{figure}

For a fixed $n$, let $V_t$ be the free $R$-module with basis all the $(n,t)$-partial diagrams. We can define an action of $B_n^R(\delta)$ on $V_t$ as follows: given a Brauer diagram $x\in B^R_n(\delta)$ and a $(n,t)$-partial digram $v\in V_t$, let the product $xv$ be obtained by placing $v$ on top of $x$, identifying the top row of $x$ with $v$, and following the edges from the bottom row of $x$. This results in a new partial diagram $w$, and a number, $j$ say, of closed loops on the top row. We then set $xv=\delta^jw$ if $w$ has exactly $t$ edges, and $xv=0$ otherwise.

Given a partition $\lambda\vdash n-2t$, we form the module $\Delta_n^R(\lambda)=V_t\otimes S^\lambda_R$, where $S^\lambda_R$ is the $R$-form of the usual Specht module (see \cite[Chapter 7]{jameskerber} for details). We then have an action of $B_n^R(\delta)$ on this module: given a Brauer diagram $x\in B^R_n(\delta)$ and a pure tensor $v\otimes s\in\Delta_n^R(\lambda)$, we define the element 
	
	\[x(v\otimes s)=(xv)\otimes\sigma(x,v)s\]
	
where $xv$ is given above and $\sigma(x,v)\in S_{n-2t}$ is the permutation on the free nodes of $xv$.

In a similar manner to previously, we then have
	\[\Delta_n^K(\lambda)=K\otimes_R\Delta_n^R(\lambda)~~~ \text{and}~~~\Delta_n^k(\lambda)=k\otimes_R\Delta_n^R(\lambda)\]
	
	\begin{rem*}
		Note that the simple modules may not have an $R$-form, so we \textbf{cannot} in general provide a module $L^R_n(\lambda)$ such that $L_n^K(\lambda)=K\otimes_R L_n^R(\lambda)$ or $L_n^k(\lambda)=k\otimes_R L_n^R(\lambda)$.
	\end{rem*}
	
As we did with the symmetric group algebra, we wish to characterise the blocks of $B_n^\bbf(\delta)$. A geometric description of these blocks in characteristic $0$ is given in \cite{geom}. A brief account of this is provided below, but is adapted to the infinite case as we will later need to consider limiting blocks.

Let $\{\varepsilon_1,\varepsilon_2,\varepsilon_3,\dots\}$ be a set of formal symbols, $p>2$ be a prime, and set
\[ \mathcal{X}=\bigoplus^\infty_{i=1}\mathbb{R}\varepsilon_i\]
We have an inner product on $\mathcal{X}$ given by extending linearly the relations
	\begin{equation}
		(\varepsilon_i,\varepsilon_j)=\delta_{ij}
		\label{eqn:innerprod}
	\end{equation}
(Here $\delta_{ij}$ is the Kronecker delta).\\
Let $\Phi=\{\pm(\varepsilon_i-\varepsilon_j),\pm(\varepsilon_i+\varepsilon_j):1\leq i<j\}$ be the infinite root system of type $D$, and $W$ the corresponding Weyl group, generated by the reflections $s_\alpha$ $(\alpha\in\Phi)$. There is an action of $W$ on $\mathcal{X}$, the generators acting by
	\[ s_\alpha(x) = x-(x,\alpha)\alpha \]
We may also define $W_p$ to be the corresponding affine Weyl group, generated by the reflections
\linebreak
  \mbox{$s_{\alpha,rp}$ $(\alpha \in \Phi,r\in\mathbb{Z})$}, with an action on $\mathcal{X}$ given by
	\[s_{\alpha,rp}(x)=x-((x,\alpha)-rp)\alpha\]
	
Fix the element
	\[\rho=\rho(\delta)=\left(-\frac{\delta}{2},-\frac{\delta}{2}-1,-\frac{\delta}{2}-2,-\frac{\delta}{2}-3,\dots\right)\]
We may then define a different (shifted) action of $W$ (resp. $W_p$) on $\mathcal{X}$ given by
	\[ w\cdot_\delta x=w(x+\rho(\delta))-\rho(\delta) \]
for all $w\in W$ (resp. $W_p$) and $x \in \mathcal{X}$.\\
		
Note that for any partition $\lambda=(\lambda_1,\lambda_2,\dots)$ there is a corresponding element $\sum\lambda_i\varepsilon_i\in \mathcal{X}$, where any $\lambda_i$ not appearing in $\lambda$ is taken to be zero. In this way we may consider partitions to be elements of $\mathcal{X}$, and write $\Lambda=\bigcup_{n\in\mathbb{N}}\Lambda_n$ for the set of all partitions.

Finally we define the transposed partition $\lambda^T$ of $\lambda$, corresponding to the Young diagram
	\[[\lambda^T]=\{(y,x)~|~(x,y)\in[\lambda]\}\]
	
From \cite{geom} we have the following results:
	
	\begin{thm}[{\cite[Theorem 4.2]{geom}}]
		Let $\lambda,\mu \in \Lambda_n$. Then the two $B_n^K(\delta)$-cell modules $\Delta_n^K(\lambda^T)$ and $\Delta^K_n(\mu^T)$ are in the same $K$-block if and only if $\mu\in W\cdot_\delta\lambda$.\label{thm:geomblocks}
	\end{thm}

and

	\begin{thm}[{\cite[Theorem 6.4]{geom}}]
		Let $\Char k=p>2$, and $\lambda,\mu \in \Lambda_n$. Then the two $B_n^k(\bar\delta)$-cell modules $\Delta^k_n(\lambda^T)$ and $\Delta^k_n(\mu^T)$ are in the same $k$-block only if $\mu\in W_p\cdot_{\bar\delta}\lambda$.
		\label{thm:orbits}
	\end{thm}
	
	\begin{rem*}
		It is important to note that the converse of Theorem \ref{thm:orbits} is not true in general. In fact, it it shown in \cite[Section 7]{geom} that counter-examples exist for arbitrarily large values of $n$.
	\end{rem*}

The next result from \cite[Section 3]{hartmannpaget} allows us to use results known for the symmetric group algebras when dealing with the Brauer algebra.

	\begin{thm}[{\cite[Proposition 3.1]{hartmannpaget}}]
		Let $\lambda,\mu\vdash n-2t$ be partitions. Then 
			\[\mathrm{Hom}_{B_n^k(\bar\delta)}(\Delta_n^k(\lambda),\Delta_n^k(\mu))\cong\mathrm{Hom}_{kS_{n-2t}}(S^\lambda_k,S^\mu_k)\]
		In particular, given two partitions $\lambda,\mu\vdash n-2t$, the two $B_n^k(\bar\delta)$-cell modules $\Delta_n^k(\lambda)$ and $\Delta_n^k(\mu)$ are in the same $k$-block if the two Specht modules $S^\lambda_k$ and $S^\mu_k$ are in the same block over the symmetric group algebra $kS_{n-2t}$.
		\label{thm:symblocks}
	\end{thm}
	
In what follows we will set $R$ to be the $p$-adic integers and identify its value group with $\mathbb{Z}$. In characteristic 0, we have from \cite{wenzl} that $B_n(\delta)$ is semisimple if $\delta\notin\mathbb{Z}$, so we will henceforth assume that $\delta\in\mathbb{Z}$. In positive characteristic, we will similarly follow the results of \cite[Section 6]{geom} and assume that $\delta$ is an element of the prime subfield $\mathbb{Z}_p\subset k$. In order to make use of Theorem \ref{thm:orbits} we will only concern ourselves with values of $p>2$.

\section{Abacus notation}

To each partition we can associate an abacus diagram, consisting of $p$ columns, known as runners, and a configuration of beads across these. By convention we label the runners left to right, starting with 0, and the positions on the abacus are also numbered from left to right, working down from the top row, starting with $0$. Given a partition $\lambda=(\lambda_1,\dots,\lambda_l)\vdash m$, fix a positive integer $b\geq m$ and construct the $\beta$-sequence of $\lambda$, defined to be
	\[\beta_\lambda=(\lambda_1-1+b,\lambda_2-2+b,\dots,\lambda_l-l+b,-(l+1)+b,\dots2,1,0)\]
Then place a bead on the abacus in each position given by $\beta_\lambda$, so that there are a total of $b$ beads across the runners. Note then that for a fixed value of $b$, the abacus is uniquely determined by $\lambda$, and any such abacus arrangement corresponds to a partition simply by reversing the above. Here is an example of such a construction.

	\begin{example}
		In this example we will fix the values $p=5,m=9,b=10$ and represent the partition $\lambda=(5,4)$ on the abacus. Following the above process, we first calculate the $\beta$-sequence of $\lambda$:
			\begin{eqnarray*}
				\beta_\lambda	&=&(5-1+10,~4-2+10,\;-3+10,\;-4+10,\dots,\;-9+10,\;-10+10)\\
							&=&(14,12,7,6,5,4,3,2,1,0)
			\end{eqnarray*}
		The next step is to place beads on the abacus in the corresponding positions. We also number the beads, so that bead $1$ occupies position $\lambda_1-1+b$, bead $2$ occupies position $\lambda_2-2+b$ and so on. The labelled spaces and the final abacus are shown below.
			\begin{figure}[H]
				\centering
				\includegraphics[scale=0.6]{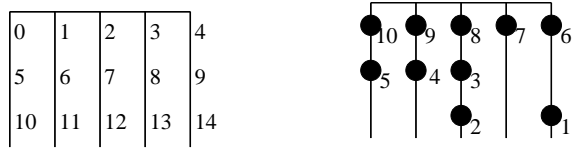}
				\caption{The positions on the abacus with 5 runners, and the arrangement of beads (numbered) representing $\lambda=(5,4)$}
			\end{figure}
		\label{ex:abacus}
	\end{example}
After fixing values of $p$ and $b$, we will abuse notation and write $\lambda$ for both the partition and the corresponding abacus with $p$ runners and $b$ beads.\\
	
We wish to investigate the effect moving the beads on the abacus has on the partitions being represented. One such result given in \cite[Chapter 2.7]{jameskerber} is as follows: Sliding a bead down (resp. up) one space on its runner corresponds to adding (resp. removing) a rim $p$-hook to the Young diagram of the partition. Therefore by sliding all beads up their runners as far as they will go, we remove all rim $p$-hooks and arrive at the $p$-core. Therefore two partitions have the same $p$-core if and only if the number of beads on corresponding runners of the two abaci is the same. Therefore as in \cite[Chapter 2.7]{jameskerber} we may re-state Nakayama's conjecture (Theorem \ref{thm:nakayama}) as:
 
	\begin{thm*}
		Two partitions $\lambda, \mu\vdash n$ label Specht modules in the same $k$-block for the symmetric group algebra $kS_n$ if and only if when represented on an abacus with $p$ runners and $b$ beads, the number of beads on corresponding runners is the same.
	\end{thm*}

Combining this with Theorem \ref{thm:symblocks} we deduce that $\lambda$ and $\mu$, partitions of $n-2t$, are in the same $B_n^k(\bar\delta)$-block if it is possible to reach one configuration of beads from the other by a sequence of moves that takes two beads, slides one $r$ places down its runner and slides the other $r$ places up. We will use the notation
	\[ \mu=a_{(i,j)}^r(\lambda)\]
to indicate that the abacus representing the partition $\mu$ is obtained from that representing $\lambda$ by sliding bead $i$ down by $r$ spaces and bead $j$ up by $r$ spaces.

\begin{rem}
Note that this notation will only be well defined after fixing the number of runners $p$ and beads $b$, which will be the case in all subsequent uses.\label{rem:arij}
\end{rem}

Following \cite[Section 7]{geom}, we now impose the condition that the number of beads $b$ used also satisfies
	\begin{equation}
		2b\equiv 2-\delta\text{ (mod $p$)}
		\label{eqn:beads}
	\end{equation}
so that $\lambda$ and $\mu$ are in the same $W_p$-orbit if and only if when using a total of $b$ beads, the number of beads on runner 0 and sum of the number of beads on runners $t$ and $p-t$ ($t>0$) is the same on both the abaci of $\lambda$ and $\mu$.
This pairing of runners suggests an alternative way of viewing the abacus.

Given a partition $\lambda$, we construct the usual abacus as above. However we then link runners $t$ and $p-t$ for each $0<t\leq \frac{p-1}{2}$ via an arc above the diagram. See Figure \ref{fig:abacus1} for an example.

	\begin{figure}[H]
		\centering
		\includegraphics[scale=0.6]{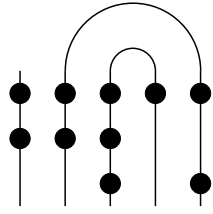}
		\caption{The partition $(5,4)$ on the new abacus}
		\label{fig:abacus1}
	\end{figure}
	
Along with sliding beads up and down runners, we can now slide them up and over the top arc. This allows us to introduce the following move: choose two beads, and slide them both $r$ places up their runners, over the arc and back down their respective paired runners. Figure \ref{fig:abacus2} shows some examples of such a move. If one of the beads lies on runner 0, that bead simply moves up and back down the runner, visiting position zero just once.\\
Note that beads may both move from left to right, right to left, or one may move left and the other right. There is also no restriction on which two beads we move, provided that they move the same number of spaces and end in an unoccupied position. Figure \ref{fig:abacus2} shows some examples of this move.

	\begin{figure}[H]
		\centering
		\includegraphics[scale=0.6]{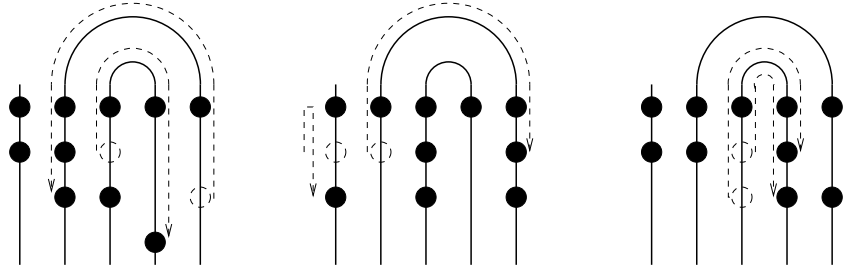}
		\caption{Moving beads on the abacus given in Figure \ref{fig:abacus1}}
		\label{fig:abacus2}
	\end{figure}

We will use the notation
	\[\mu=d^r_{(i,j)}(\lambda)\]
to indicate that the abacus representing the partition $\mu$ is obtained from that representing $\lambda$ by sliding beads $i$ and $j$ both by $r$ spaces up and over the arc to their paired runner (or up and down if on runner 0). 

Note that the comments in Remark \ref{rem:arij} also apply here.

The next section shows how this move relates the partitions as labels for cell modules.

\section{The blocks of the Brauer algebra}

Recall the localisation and globalisation functors, $F_n$ and $G_n$ respectively, and the indexing set $\Lambda_n$ of $B_n^\bbf(\delta)$ cell modules (see \eqref{eq:fn}, \eqref{eq:gn} \& \eqref{eq:indexset}). The functors give us a full embedding of $B_n^\bbf(\delta)$-\catmod \;inside $B_{n+2}^\bbf(\delta)$-\catmod, and hence an embedding of $\Lambda_n$ inside $\Lambda_{n+2}$. Fix the value of the parameter $\delta$ and let $\mathcal{B}_n^\bbf(\lambda)$ denote the set of partitions labelling cell modules in the $B_n^\bbf(\delta)$-block containing $\Delta_n^\bbf(\lambda)$. We can then embed $\mathcal{B}_n^\bbf(\lambda)$ inside $\mathcal{B}_{n+2}^\bbf(\lambda)$, and consider the limiting set
	\[ \mathcal{B}^\bbf(\lambda)=\bigcup_{\substack{n\in\mathbb{N}\\n\equiv|\lambda|\text{ mod }2}}\mathcal{B}_n^\bbf(\lambda)\subset\Lambda \]
Recall the result from \cite{geom} that the orbits of the affine Weyl group of type $D$ on the set of partitions correspond to unions of blocks of the Brauer algebra. Our aim is to use the moves $a_{(i,j)}^r$ and $d_{(i,j)}^r$ to show that in the limiting case, the orbits and blocks are equal.

We begin by proving the following general result:

	\begin{lemma}
		Suppose $X,Y$ are $R$-free $A$-modules of finite rank and let $M\subseteq KY$. If $\mathrm{Hom}_{K}(KX,KY/M)\neq0$ then there is a submodule $N\subseteq kY$ such that $\mathrm{Hom}_{k}(kX,kY/N)\neq0$. Moreover, $N$ is the $p$-modular reduction of a lattice in $M$.
		\label{lem:modred}
	\end{lemma}
	
		\begin{proof}
			Let \mbox{$Q=KY/M$} be the image of the canonical quotient map
			\mbox{$\rho:KY\rightarrow KY/M$}, and let 
			\linebreak
			$f\in\mathrm{Hom}_{K}(KX,Q)$ be non-zero. Note that $\rho(Y)$ is a lattice in $Q$, since $Y$ has finite rank and
			\linebreak 
			\mbox{$K\rho(Y)=\rho(KY)=Q$}. As $X$ and $Y$ are modules of finite rank we may assume that 
				\begin{equation}
					f(X)\subseteq \rho(Y)\text{~~~but~~~}f(X)\nsubseteq \pi \rho(Y)
					\label{eqn:pires}
				\end{equation}
for instance by considering the matrix of $f$ and multiplying the coefficients by an appropriate power of $\pi$.\\
Then $f$ restricts to a homomorphism $X\rightarrow\rho(Y)$, and induces a homomorphism $\overline{f}:kX\rightarrow k\rho(Y)$. This must be non-zero since we can find $x\in X$ such that $f(x)\in \rho(Y)\backslash\pi \rho(Y)$ by \eqref{eqn:pires}.\\

			It remains to prove that $k\rho(Y)$ can be taken to be $kY/N$ for some $N\subset kY$, the modular reduction of a lattice in $M$. We have the following maps:

				\begin{figure}[H]
					\[\xymatrix{0\ar[r]&M\ar[r]&KY\ar@{->>}[r]^\rho&Q&KX\ar[l]_f\\
	        				 0\ar[r]&L\ar[r]&Y\ar@{->>}[r]\ar@{}[u]|\cup\ar@{->>}[d]&\rho(Y)\ar@{}[u]|\cup\ar@{->>}[d]&X\ar@{}[u]|\cup\ar@{->>}[d]\\
	       				  &&kY\ar@{->>}[r]&k\rho(Y)&kX\ar[l]_<<<<<{\overline{f}}}\]
				\end{figure}
			where $L=\mathrm{Ker}(Y\longrightarrow\rho(Y))$\\
			\\
			The $K$-module $Q$ is torsion free. Therefore as an $R$-module, $\rho(Y)\subseteq Q$ must also be torsion free. Since $R$ is a principal ideal domain (by definition of it being a discrete valuation ring), the structure theorem for modules over a PID tells us that $\rho(Y)$ must be free. It is therefore projective, and the exact sequence
				\[0\longrightarrow L\longrightarrow Y\longrightarrow \rho(Y)\longrightarrow0\]
			is split.
			Then since the functors $K\otimes_R-$ and $k\otimes_R-$ preserve split exact sequences, we deduce that $M\cong KL$ and we can set $N=kL$ to complete the exact sequence
				\[0\longrightarrow N\longrightarrow kY\longrightarrow k\rho(Y)\longrightarrow0\]
satisfying the requirements above.
		\end{proof}
		
Using this, we may then prove:

	\begin{propn}
		If $\mu\in\mathcal{B}^K_n(\lambda)$, then $\mu\in\mathcal{B}^k_n(\lambda)$.
		\label{prop:blocks}
	\end{propn}
	
		\begin{proof}
			By the cellularity of $B_n^K(\delta)$, as detailed in \cite{grahamlehrer}, partitions $\lambda$ and $\mu$ are in the same $K$-block if and only if there is a sequence of partitions
				\[ \lambda=\lambda^{(1)},\lambda^{(2)},\dots,\lambda^{(t)}=\mu \]
			and $B_n^K(\delta)$-modules
				\[ M^{(i)}\leq\Delta_n^K(\lambda^{(i)}) ~~~~~ (1<i\leq t) \]
			such that for each $1\leq i<t$
				\[ \mathrm{Hom}_{B_n^K(\delta)}(\Delta_n^K(\lambda^{(i)}),\Delta_n^K(\lambda^{(i+1)}/M^{(i+1)})\neq0 \]
			The application of Lemma \ref{lem:modred} then shows
				\[ \mathrm{Hom}_{B_n^k(\bar\delta)}(\Delta_n^k(\lambda^{(i)}),\Delta_n^k(\lambda^{(i+1)}/\overline{M^{(i+1)}})\neq0 \]
			giving us such a sequence of partitions linking $\lambda$ and $\mu$, except now we are working with $B_n^k(\bar\delta)$-modules.
		\end{proof}
		
Moreover, we have an interpretation of $d^r_{(i,j)}$ as the action of an element of a Weyl group:

	\begin{lemma}
		If two partitions $\lambda,\mu\in\Lambda_n$, both represented with $b$ beads on an abacus with $p$ runners, are related by the move $d_{(i,j)}^r(\lambda)=\mu$, then \mbox{$\mu\in W\cdot_{\delta'}\lambda$} for $\delta'=rp-2b+2$.
		\label{lem:moveweyl}
	\end{lemma}
	
		\begin{proof}
			Suppose $\lambda=(\lambda_1,\dots,\lambda_l)$, then
				\[\beta_\lambda=(\lambda_1-1+b,\lambda_2-2+b,\dots,\lambda_l-l+b,-(l+1)+b,\dots,0) \]
			The move $d^r_{(i,j)}$ can be represented in the $\beta$-sequence. Indeed, we can decompose $r=r_1+1+r_2$, where $r_1$ is the number of positions needed to move to the top of the runner, then $1$ for over the arc, and finally $r_2$ positions down. Then:
				\begin{enumerate}
					\item Sliding bead $i$ up by $r_1$ spaces puts that bead in position $\lambda_i-i+b-r_1p$.
					\item Moving across to the paired runner places it in position $p-(\lambda_i-i+b-r_1p)$.
					\item Finally, sliding it down by $r_2$ spaces lands it in position
						\[p-(\lambda_i-i+b-r_1p)+r_2p = -\lambda_i+i-b+rp\]
				\end{enumerate}
			A similar process leads to the same result if the bead is on runner zero: we just write $r=r_1+r_2$ and omit Step 2.

			This leaves us with
				\begin{equation*}
					d^r_{(i,j)}(\beta_\lambda)=(\lambda_1-1+b,\dots,-\lambda_i+i-b+rp,\dots,-\lambda_j+j-b+rp,\dots,0)
				\end{equation*}
			Note that this will not be the $\beta$-sequence of a partition, i.e. will not be a strictly decreasing sequence. However by re-arranging the sequence we see that setwise it must be that of $\beta_\mu$, since the beads occupy the same positions in the abacus. Therefore, for an appropriate element $w\in S_b$ of the symmetric group on $b$ letters, we have:
				\[w^{-1}(\beta_\mu)=(\lambda_1-1+b,\dots,\underbrace{-\lambda_j+j-b+rp}_\text{$i$-th place},\dots,\underbrace{-\lambda_i+i-b+rp}_\text{$j$-th place},\dots)\]
			and hence
				\begin{align}
					\beta_\lambda-w^{-1}\beta_\mu&=(\lambda_i+\lambda_j-(i+j)+2b-rp)(\varepsilon_i+\varepsilon_j)\notag\\
											&=(\lambda+\rho(rp-2b+2),\varepsilon_i+\varepsilon_j)(\varepsilon_i+\varepsilon_j)\notag\\
											&=(\lambda+\rho(\delta'),\varepsilon_i+\varepsilon_j)(\varepsilon_i+\varepsilon_j)
					\label{eqn:wbeta}
				\end{align}
			making the substitution $\delta'=rp-2b+2$ and using the inner product defined in \eqref{eqn:innerprod}.\\
			\\
			Now define
				\[\eta=(1,2,3,\dots,b),~~~~~\omega=(1,1,\dots,1)\]
			We may then rewrite \eqref{eqn:wbeta} as
				\[w^{-1}(\mu-\eta+b\omega)=\lambda-\eta+b\omega-(\lambda+\rho(\delta'),\varepsilon_i+\varepsilon_j)(\varepsilon_i+\varepsilon_j)\]
			and noticing that $\omega$ is invariant under the action of $S_b$:
				\begin{align*}
					&w^{-1}(\mu-\eta)=\lambda-\eta-(\lambda+\rho(\delta'),\varepsilon_i+\varepsilon_j)(\varepsilon_i+\varepsilon_j)\\
					\implies&	w^{-1}\left(\mu+\rho(\delta')+\left(\frac{\delta'}{2}-1\right)\omega\right)
					=\lambda+\rho(\delta')+\left(\frac{\delta'}{2}-1\right)\omega-(\lambda+\rho(\delta'),\varepsilon_i+\varepsilon_j)(\varepsilon_i+\varepsilon_j)\\
					\implies&w^{-1}(\mu+\rho(\delta'))=\lambda+\rho(\delta')-(\lambda+\rho(\delta'),\varepsilon_i+\varepsilon_j)(\varepsilon_i+\varepsilon_j)\\
					\implies&\mu=ws_{\varepsilon_i+\varepsilon_j}(\lambda+\rho(\delta'))-\rho(\delta')
				\end{align*}
			If we write $w=(i_1~j_1)(i_2~j_2)\dots(i_t~j_t)$ as a product of transpositions, then the action of this element is the same as that of $s_{\varepsilon_{i_1}-\varepsilon_{j_1}}s_{\varepsilon_{i_2}-\varepsilon_{j_2}}\dots s_{\varepsilon_{i_t}-\varepsilon_{j_t}}\in W$. So we may assume that $w\in W$.
			Therefore we may obtain $\mu$ by
				\[\mu=ws_{\varepsilon_i+\varepsilon_j}\cdot_{\delta'}\lambda\]
			where $ws_{\varepsilon_i+\varepsilon_j}\in W$. 
		\end{proof}
		
	\begin{cor}
If two partitions $\lambda,\mu\in\Lambda_n$, both represented with $b$ beads on an abacus with $p$ runners, are related by a single move $a_{(i,j)}^r$ or $d_{(i,j)}^r$, then $\mu^T\in\mathcal{B}_n^k(\lambda^T)$
		\label{cor:singmove}
	\end{cor}
	
		\begin{proof}
			If $\mu=a^r_{(i,j)}(\lambda)$ then we apply Nakayama's Conjecture (Theorem \ref{thm:nakayama}) and Theorem \ref{thm:symblocks} to show this result.

			If $\mu=d^r_{(i,j)}(\lambda)$, then by Lemma \ref{lem:moveweyl}
				\[\mu\in W\cdot_{\delta'}\lambda\]
			where $\delta'=rp-2b+2$.
			By Theorem \ref{thm:geomblocks}, we have that $\lambda^T$ and $\mu^T$ are in the same $B_n^K(\delta')$-block. Proposition \ref{prop:blocks} then shows that they must then be in the same $B_n^k(\bar\delta)$-block, since our condition on $b$
				\[2b\equiv2-\delta\text{ (mod $p$)}\] implies $\delta'\equiv\delta$ (mod $p$), so also reduces to $\bar\delta$.
		\end{proof}
		
	\begin{cor}
		If two partitions $\lambda,\mu\in\Lambda_n$, both represented with $b$ beads on an abacus with $p$ runners, are related by a sequence of moves $a_{(i,j)}^r$ and $d_{(i,j)}^r$, then $\mu^T\in\mathcal{B}^k(\lambda^T)$
		\label{cor:red}
	\end{cor}

		\begin{proof}
			Let
				\[ \lambda=\lambda^{(1)},\lambda^{(2)},\dots,\lambda^{(t)}=\mu \]
			be a sequence of partitions such that

				\[\lambda^{(l+1)}=	\begin{cases}
									a^{r_l}_{(i_l,j_l)}\lambda^{(l)}~~~\text{or}~~~ \\
									d^{r_l}_{(i_l,j_l)}\lambda^{(l)}
  								\end{cases}\]
			for some values $r_l\in\mathbb{Z}$ and $1\leq i_l<j_l\leq b$ $(1\leq l<t)$.\\
			Then by applying Corollary \ref{cor:singmove} we have \mbox{$\lambda^{(l+1)}\in\mathcal{B}_{n_l}^k(\lambda^{(l)})$}, where 
			$n_l=\text{max}(|\lambda^{(l)}|,|\lambda^{(l+1)}|)$. Therefore \mbox{$\mu\in\mathcal{B}^k(\lambda)$}.
		\end{proof}
		
	\begin{rem*}
		In Corollary \ref{cor:red}, the values of $n_l$ may exceed $n$, and hence the two partitions may not be in the same block when we restrict to $\mathcal{B}_n^k(\lambda^T)$.
	\end{rem*}
	
	\begin{example}
		Recall the abacus in Figure \ref{fig:abacus1} and the leftmost one in Figure \ref{fig:abacus2}. They represent partitions $\lambda=(5,4)$ and $\mu=(9,4^2)$ respectively, labelling $B_n^k(\bar2)$-cell modules over a field $k$ of characteristic 5. They are linked by the move $\mu=d^5_{(1,3)}(\lambda)$, and we will show that $\mu\in\mathcal{B}_{17}^k(\lambda)$. 

		Indeed, if we set $\delta'=rp-2b+2=25-20+2=7$, then 
			\begin{align*}
				s_{\varepsilon_1+\varepsilon_3}\cdot_{\delta'}\lambda&=s_{\varepsilon_1+\varepsilon_3}(\lambda+\rho(\delta'))-\rho(\delta')\\
					&=\lambda-(\lambda+\rho(\delta'),\varepsilon_1+\varepsilon_3)(\varepsilon_1+\varepsilon_3)\\
					&=(5,4,0,0,\dots)-\left(\left(\frac{3}{2},-\frac{1}{2},-\frac{11}{2},-\frac{13}{2},\dots\right),\varepsilon_1+\varepsilon_3\right)(\varepsilon_1+\varepsilon_3)\\
					&=(5,4,0,0,\dots)+4(\varepsilon_1+\varepsilon_3)\\
					&=(9,4,4,0,0,\dots)=\mu
			\end{align*}
		Therefore the transposed partitions are in the same $B_{17}^K(7)$-block in characteristic 0. Taking the modular reduction, we have the desired result as \mbox{$7\equiv2$ (mod 5)}.
	\end{example}
	
We now want to show that given two partitions in the same $W_p$-orbit, their abaci can be related by a series of moves $a^r_{(i,j)}$ and $d^r_{(i,j)}$, and are therefore in the same $B_n^k(\bar\delta)$-block (for some $n$). We do this by defining a ``$b$-reduced abacus'', and show that we can arrive at it from our partition using only moves $a^r_{(i,j)}$ and $d^r_{(i,j)}$.

	\begin{definition*}
		An abacus with $p$ runners is called \emph{$b$-reduced} if it contains $b$ beads, all of which are on runners $0$ to $\frac{p-1}{2}$, and all beads are as high up on their runners as possible, except for the last bead on runner $0$ which may be one space down.\\
		For a fixed prime $p$, we say a partition is $b$-reduced if when represented using $b$ beads on an abacus with $p$ runners, that abacus is $b$-reduced.
	\end{definition*}
	
	\begin{figure}[H]
		\centering
		\includegraphics[scale=0.6]{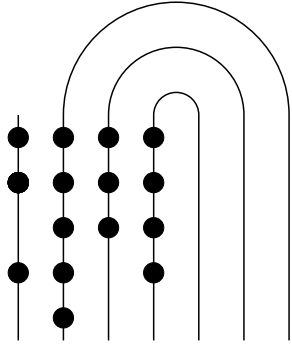}
		\caption{An example of a reduced abacus}
	\end{figure}
	
Given the abacus arrangement of a partition $\lambda$ with $b$ beads, we define the \emph{$b$-reduction of $\lambda$}, written $\overline{\lambda}^b$, to be the $b$-reduced abacus satisfying:\\
\begin{minipage}{1.0\textwidth}
	\begin{itemize}
		\item The number of beads on runner $0$ of $\lambda$ and $\overline{\lambda}^b$ is equal.
		\item For each $1\leq t\leq \frac{p-1}{2}$, the number of beads on runner $t$ of $\overline{\lambda}^b$ is equal to the sum of the number of beads on runners $t$ and $p-t$ of $\lambda$.
		\item $|\overline{\lambda}^b|-|\lambda|\in2\mathbb{Z}$
	\end{itemize}
\end{minipage}
	
	\begin{rem*}
		This latter condition determines whether or not the final bead on runner 0 is moved down a space or not, and thus ensures that the reduction is unique.
	\end{rem*}
	
Example \ref{ex:reduction} below shows such a construction.\\
\\
For a fixed value of $b$ (satisfying the congruence condition \eqref{eqn:beads}), the $W_p$-orbits on partitions represented with $b$ beads are characteristed by the number of beads on runner 0 and the  total number of beads on pairs of runners $t$ and $p-t$ for each $t>0$. As a result, each $W_p$-orbit corresponds to a unique $b$-reduced partition, obtained by taking the reduction of any partition in the orbit that can be represented on an abacus with $b$ beads.

We will now describe an algorithm for constructing $\overline{\lambda}^b$ from $\lambda$, using only combinations of $a^r_{(i,j)}$ and $d^r_{(i,j)}$. In order to make this process more understandable, we use these basic moves to build 4 others which makes the manipulation of the abacus easier:
	\begin{itemize}
		\item[(M1)] We can slide a bead one space up a runner, provided we move another bead down a space. 
	
			This is simply the move $a_{(i,j)}^1$.
			
				\begin{figure}[H]
					\centering
					\includegraphics[scale=0.6]{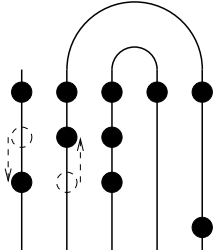}
					\caption{An example of move (M1)}
				\end{figure}

		\item[(M2)] We can move a bead and the one directly below on the same runner together over the top of the abacus to any (unoccupied) position on the paired runner.
	
			This is the move $d_{(i,j)}^r$, where $i$ and $j$ are consecutive beads on the same runner.
			
				\begin{figure}[H]
					\centering
					\includegraphics[scale=0.6]{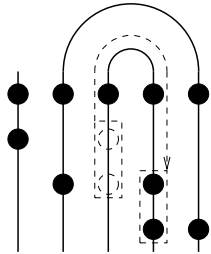}
					\caption{An example of move (M2)}
				\end{figure}

	\item[(M3)] Provided there is a runner with at least two beads on it, we can move any two beads each one space up their runner, or one bead up two spaces.
	
			We do this by combining (M2) and (M1). Suppose we wish to move beads $i$ and $j$ each up by one space. We first choose the runner with at least two beads on it, and let the last two (those lowest down) be $x$ and $y$, ensuring that neither are equal to $i$ or $j$.
			
			We then apply (M1) twice, specifically the moves that push bead $i$ up and $x$ down, then $j$ up and $y$ down. Next we perform $d^r_{(x,y)}$ followed by $d^{r-1}_{(x,y)}$ (for a sufficiently large value of $r$), so that the beads $x$ and $y$ finish in their original places (before the first use of (M1)). Note that we are not re-labelling beads as we move them.
						
				\begin{figure}[H]
					\centering
					\includegraphics[scale=0.5]{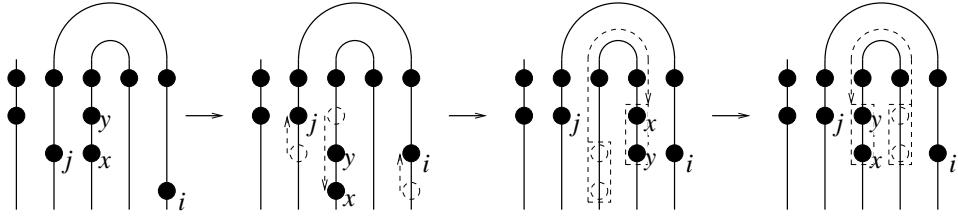}
					\caption{An example of move (M3)}
				\end{figure}
				
			\begin{rem*}
				To simplify this move, we have insisted that $x$ and $y$ are distinct from $i$ and $j$. However we can make modifications to remove this restriction. If we follow the process as given, then after an application of (M1) there will be a bead that moves up and back to its original place, and we can then apply (M2) as usual. An example is shown below in Figure \ref{fig:move3alt}. In all practical purposes however, there will usually be enough beads to allow us to use the move as originally stated.
				\begin{figure}[H]
					\centering
					\includegraphics[scale=0.5]{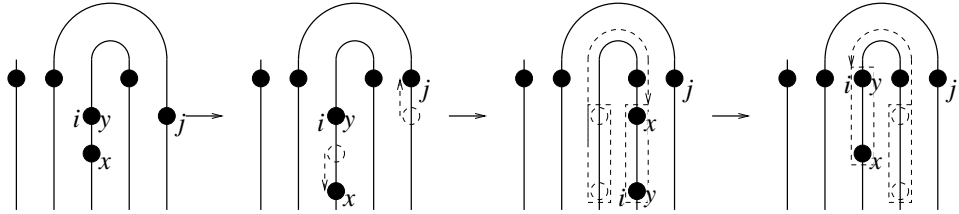}
					\caption{An example of move (M3) with beads $i$ and $y$ equal. Note that in step 2 the bead $i=y$ does not move.}
					\label{fig:move3alt}
				\end{figure}
			\end{rem*}

	\item[(M4)] Provided there is a runner with at least two beads on it, we can move any two beads each one space down their runner, or one bead down two spaces.
	
			This is simply the reverse of the move (M3).
			
	\end{itemize}
	
Now that we have detailed all the necessary moves, we describe an algorithm for obtaining the $b$-reduction of a given partition:

	\begin{enumerate}
		\item Construct the abacus of the partition in the usual way, choosing a large enough value of $b$ so that there are at least 3 beads on runner 0.
		
		\item Using (M3), we may move beads up in pairs so that each one is as high on the runner as it will go. Since we have chosen $b$ so that there are always at least 3 beads on runner 0, we may use this runner when applying (M3). Now there may be a single bead with a space above it. If this is the case, we apply (M1) and move this bead up and the last bead on runner 0 down.
		
		\item Applying (M2) repeatedly, we then slide beads in pairs over from the right side to the left, so that there is at most one bead on the right hand runners. 
		
		\item If there are no more beads on runners $\frac{p+1}{2}$ to $p-1$, then we have the reduced abacus. So assume that there is a bead on runner $t$ in this range. Pair the beads consecutively on runner $p-t$, starting from the one furthest down, and move each pair down by two spaces with (M4). This ensures an empty space in the second position on runner $p-t$ (the first may or may not be filled).
		
		\item Since there are at least 3 beads on runner 0 there must be one in the second position. Move this and the single bead on runner $t$ by 2 spaces over the arc (this is just the move $d^2_{(i,j)}$). The bead on runner 0 returns to the same space, but there are now no beads on runner $t$.
		
		\item If there is an empty space in the first position of runner $p-t$, we use (M3) to slide all the beads up in pairs to fill this gap. If the total number of beads now on the runner is odd, then the final bead cannot be moved in a pair. We then have one of the two following cases:
		
			\begin{enumerate}
				\item The last bead on runner 0 has a space before it. In this case, we can use (M3) to move this and the last bead on runner $t$ each up by one space.
				
				\item The last bead on runner 0 does not have a space before it. In this case, we use (M1) to slide the bead on runner $t$ up and the bead on runner 0 down each by one space.
			\end{enumerate}

		\item Repeat Steps 4-6 until there are no more beads on runners $\frac{p+1}{2}$ to $p-1$. We then have a reduced abacus.
	\end{enumerate}
	
	\begin{example}
		Setting $p=7$ we construct the $b$-reduction of $\lambda=(5^2,4^2,3,2^3,1)$ following the steps above. We describe the process below, and show this on the abacus in Figure \ref{fig:reduction}.
			\begin{enumerate}
				\item Choosing $b=22$ ensures that we satisfy the requirement that there are at least $3$ beads on runner $0$.
				\item We can slide the last beads on runners 5 and 6 up (M3), but we must then move the last beads on runner 0 down one space and the last bead on runner 0 up one space (M1).
				\item After sliding the beads over in pairs with (M2), we are left with a single bead at the top of runner 5.
				\item In pairs, the beads are moved down runner 2 by two spaces each with (M4), leaving 2 empty spaces at the top.
				\item The second bead on runner 0 and the bead on runner 5 are each moved by 2 spaces up and over (the move $d^2_{(i,j)}$), so that there are now no beads on runner 5.
				\item In pairs, the beads are moved up runner 2 by one space each with (M3). This leaves a single bead at the top with an empty space before it. Therefore we use (M3) and push each up one space so that we arrive at the reduced abacus.
			\end{enumerate}
			
			\begin{figure}[H]
				\centering
				\includegraphics[scale=0.473]{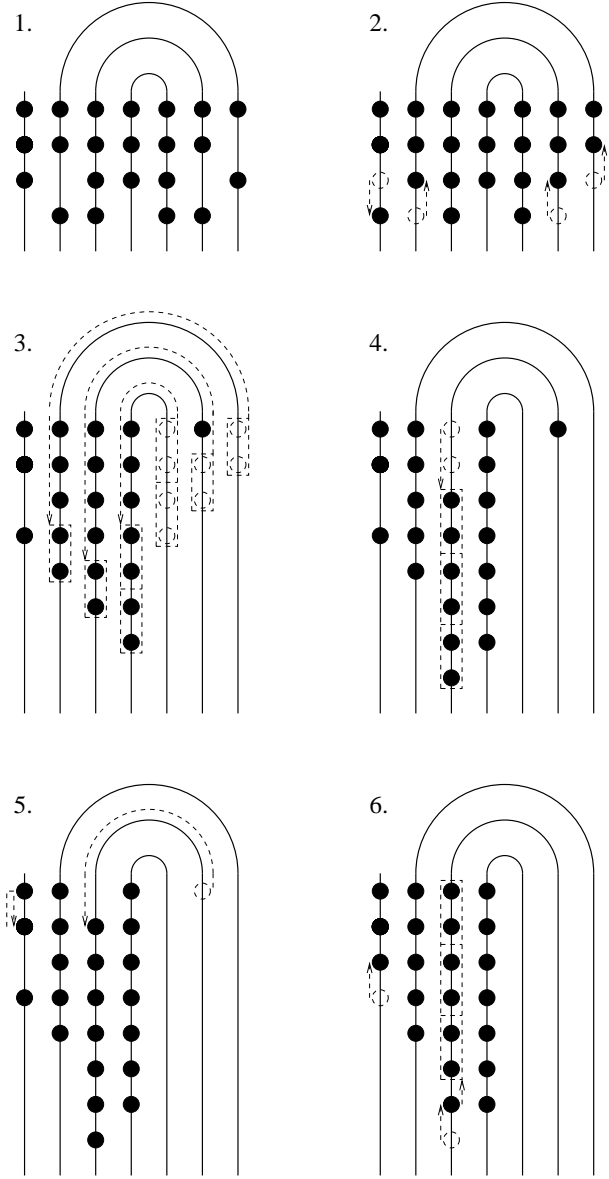}
				\caption{Constructing the $b$-reduction of $(5^2,4^2,3,2^3,1)$}
				\label{fig:reduction}
			\end{figure}
			
		\label{ex:reduction}
	\end{example}
	
We may now state the main theorem. Recall the limiting block $\mathcal{B}^k(\lambda)$ for $B_n^k(\bar\delta)$ containing $\lambda$:

	\begin{thm}
		Let $\lambda,\mu \in \Lambda$. We have $\mu^T\in\mathcal{B}^k(\lambda^T)$ if and only if $\mu\in W_{p}\cdot_{\bar\delta}\lambda$.
		In other words, $\lambda$ and $\mu$ are in the same $W_p$-orbit under the $\bar\delta$-shifted action if and only if they label cell modules in the same $B_n^k(\bar\delta)$-block for some $n$ (and hence all $m\geq n$).
		\label{thm:limorbs}
	\end{thm}
	
		\begin{proof}
			If $\mu^T\in\mathcal{B}^k(\lambda^T)$ then $\mu\in W_p\cdot_{\bar\delta}\lambda$ by Theorem \ref{thm:orbits}.

			Conversely, given $\mu\in W_{p}\cdot_{\bar\delta}\lambda$ choose a $b\geq\text{max}(|\lambda|,|\mu|)$ satisfying both the congruence \eqref{eqn:beads} and the requirement that there are at least 3 beads on runner 0 of both abaci of the partitions. Then the $b$-reductions $\overline{\lambda}^b$ and $\overline{\mu}^b$ of $\lambda$ and $\mu$ must coincide, since they have the same number of beads on corresponding pairs of runners. Therefore it is possible to reach the abacus arrangement of $\mu$ from that of $\lambda$ using a sequence of moves $a^r_{(i,j)}$ and $d^r_{(i,j)}$, travelling via the reduced abacus. Corollary \ref{cor:red} then gives the result.
		\end{proof}

\section{Homomorphisms between cell modules}

The results of the previous section give us the limiting blocks of the Brauer algebra, but do not provide any details about the structure of the blocks, in particular the composition factors of cell modules or homomorphisms between them. This section will give new results regarding this.

As mentioned in Section 2 we may view partitions as points in a Euclidean space, and the (affine) Weyl group acts as a group of isometries of this. In particular, the elements $s_{\alpha,rp}$ $(\alpha\in\Phi)$ correspond to reflections through hyperplanes. Let $\lambda$ and $\mu$ be partitions that are related via a reflection in such a hyperplane. We will show that we can assume without loss of generality that $\mu\subseteq\lambda$ or $\mu\trianglelefteq\lambda$ (where $\trianglelefteq$ denotes the dominance order on partitions, see \cite[Section 1]{jameskerber} for details). Moreover, in that case there is a non-zero homomorphism from $\Delta_n^\bbf(\lambda^T)$ to $\Delta_n^\bbf(\mu^T)$ . In particular, this will show that
	\[ [\Delta_n^\bbf(\mu^T):L_n^\bbf(\lambda^T)]\neq0 \]
	whenever the simple module $L_n^\bbf(\lambda^T)$ exists.
\begin{propn}
	Let $\lambda,\mu\in\Lambda_n$. If there is a reflection $s_{\varepsilon_i+\varepsilon_j,rp} \in W_p$ $(1\leq i<j\leq n)$ such that
			\[s_{\varepsilon_i+\varepsilon_j,rp} \cdot_{\delta} \lambda = \mu\]
	then $\mu=\lambda-(\lambda_i+\lambda_j-\delta-rp-i-j-2)(\varepsilon_i+\varepsilon_j)$. In particular, either $\mu\subseteq\lambda$ or $\lambda\subseteq\mu$.
	\label{prop:typed}
\end{propn}

	\begin{proof}
	We have
		\begin{eqnarray*}
			s_{\varepsilon_i+\varepsilon_j,rp}\cdot_\delta\lambda &=& s_{\varepsilon_i+\varepsilon_j,rp}(\lambda+\rho(\delta))-\rho(\delta)\\
				&=& \lambda+\rho(\delta)-((\lambda+\rho(\delta),\varepsilon_i+\varepsilon_j)-rp)(\varepsilon_i+\varepsilon_j)-\rho(\delta)\\
				&=& \lambda-
						\left(\lambda_i+\lambda_j-\frac{\delta}{2}-(i-1)-\frac{\delta}{2}-(j-1)-rp\right)(\varepsilon_i+\varepsilon_j)\\
				&=& \lambda-(\lambda_i+\lambda_j-\delta-rp-i-j+2)(\varepsilon_i+\varepsilon_j)
		\end{eqnarray*}
		If $(\lambda_i+\lambda_j-\delta-rp-i-j+2)\geq0$ then $\mu\subseteq\lambda$, whereas if $(\lambda_i+\lambda_j-\delta-rp-i-j+2)\leq0$ then $\lambda\subseteq\mu$.
	\end{proof}

Recall that in the Young diagram of a partition $\lambda$, the \emph{content} of a node $(x,y)$ is the value of $y-x$.

\begin{definition}[{\cite[Definition 4.7]{blocks}}]
	Two partitions $\mu\subseteq\lambda$ are \emph{$\delta$-balanced} if: (i) there exists a pairing of the nodes in $[\lambda]\backslash[\mu]$ such that the contents of each pair sum to $1-\delta$, and (ii) if $\delta$ is even and the nodes with content $-\frac{\delta}{2}$ and  $\frac{2-\delta}{2}$ in $[\lambda]\backslash[\mu]$ are configured as in Figure \ref{fig:balanced} below then the number of columns in this configuration is even.
	\begin{figure}[H]
	\centering
	\includegraphics[scale=0.6]{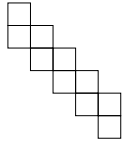}
	\caption{A possible configuration of the boxes of content $-\frac{\delta}{2}$ and $\frac{2-\delta}{2}$ in $[\lambda]\backslash[\mu]$}
	\label{fig:balanced}
	\end{figure}
	We say two partitions $\mu\subseteq\lambda$ are \emph{maximal $\delta$-balanced} if $\mu$ and $\lambda$ are $\delta$-balanced and there is no partition $\nu$ such that $\mu\subsetneq\nu\subsetneq\lambda$ with $\nu$ and $\lambda$ $\delta$-balanced.
	\label{def:balanced}
\end{definition}

From \cite{blocks}, we also have:

\begin{thm}[{\cite[Theorem 6.5]{blocks}}]
	If $\mu\subseteq\lambda$ are maximal $\delta$-balanced partitions, then 
	\[\mathrm{Hom}_{B_n^K(\delta)}(\Delta_n^K(\lambda),\Delta_n^K(\mu))\neq0\]\label{thm:maxbal}
\end{thm}

In order to make use of this result, we now prove the following.

\begin{propn}
	Let $\lambda,\mu\in\Lambda_n$. If there is a reflection $s_{\varepsilon_i+\varepsilon_j} \in W$ $(1\leq i<j\leq n)$ such that
			\[s_{\varepsilon_i+\varepsilon_j} \cdot_{\delta} \lambda = \mu\]
		with $\mu\subseteq\lambda$, then $\mu^T$ and $\lambda^T$ are maximal $\delta$-balanced.
		\label{prop:maximal}
\end{propn}
	\begin{proof}
		We first show that $\mu^T$ and $\lambda^T$ are $\delta$-balanced. For more details, see the proof of \cite[Theorem 4.2]{geom}.
		
		From Proposition \ref{prop:typed} we have $\mu=\lambda-(\lambda_i+\lambda_j-\delta-i-j+2)(\varepsilon_i+\varepsilon_j)$. If $\mu=\lambda$ there is nothing to prove, so we will assume that $\mu\subsetneq\lambda$ and see that $[\lambda]\backslash[\mu]$ consists of two strips of nodes in rows $i$ and $j$. The content of the last node in row $i$ of $[\mu]$ is given by
		\begin{align*}
			\mu_i-i&=\lambda_i-(\lambda_i+\lambda_j-\delta-i-j+2)-i\\
					&=-\lambda_j+j+\delta-2
		\end{align*}
		Therefore the content of the first node in row $i$ of $[\lambda]\backslash[\mu]$ is $\mu_i-i+1=-\lambda_j+j+\delta-1$, and so the nodes in row $i$ of $[\lambda]\backslash[\mu]$ have content
		\[-\lambda_j+j+\delta-1,-\lambda_j+j+\delta,\dots,\lambda_i-(i-2),\lambda_i-(i-1),\lambda_i-i\]
		Similarly, the nodes in row $j$ of $[\lambda]\backslash[\mu]$ have content
		\begin{equation}-\lambda_i+i+\delta-1,-\lambda_i+i+\delta,\dots,\lambda_j-(j-2),\lambda_j-(j-1),\lambda_j-j\label{eqn:pair}\end{equation}
		If we pair these two rows in reverse order, the contents of each pair sum to $\delta-1$. If we take the transpose of the partitions, we then have a pairing of two columns of nodes such that the content of each pair sums to $1-\delta$, satisfying condition (i) of Definition \ref{def:balanced} above. Moreover, since after taking the transpose we are always pairing boxes in different columns, condition (ii) is also satisfied. Hence $\mu^T$ and $\lambda^T$ are $\delta$-balanced.\\
		\\
		Suppose now there is a partition $\nu$ such that $\mu^T\subseteq\nu^T\subsetneq\lambda^T$ with $\lambda^T$ and $\nu^T$ $\delta$-balanced. Then after transposing, $[\lambda]\backslash[\nu]$ consists of nodes from rows $i$ and $j$ and, since $\nu\neq\lambda$, must contain at least one of the final nodes in row $i$ or $j$, say row $i$ (the case of row $j$ is similar).
		
	This final node has content $\lambda_i-i$ and, since $\lambda^T$ and $\nu^T$ are $\delta$-balanced, must be paired with a node of content $\delta-1-(\lambda_i-i)$. But using \eqref{eqn:pair} above, and the fact that $i<j$, we have that the only such node in row $i$ or $j$ of $[\lambda]\backslash[\mu]$ is the first in row $j$. As $[\lambda]\backslash[\nu]$ now must contain the first node in row $j$ of $[\lambda]\backslash[\mu]$, it contains all nodes in row $j$ of $[\lambda]\backslash[\mu]$, in particular the final node in this row.
	
	By repeating the argument of the above paragraph, we see that $[\lambda]\backslash[\nu]$ also contains all the nodes in row $i$ of $[\lambda]\backslash[\mu]$, hence $\nu=\mu$. Therefore $\nu^T=\mu^T$ and we deduce that $\mu^T$ and $\lambda^T$ are maximal $\delta$-balanced.
\end{proof}

We may now show:

\begin{thm}
		Let $\lambda,\mu\in\Lambda_n$. If there is a reflection $s_{\varepsilon_i+\varepsilon_j,rp} \in W_p$ $(1\leq i<j\leq n)$ such that
			\[s_{\varepsilon_i+\varepsilon_j,rp} \cdot_{\delta} \lambda = \mu\]
		then without loss of generality $\mu\subseteq\lambda$ and $\mathrm{Hom}_{B_n^k(\bar\delta)}(\Delta_n^k(\lambda^T),\Delta_n^k(\mu^T))\neq0$. In particular, if $\lambda^T$ is $p$-regular we have a non-zero decomposition number $[\Delta_n^k(\mu^T):L_n^k(\lambda^T)]\neq0$
\end{thm}

	\begin{proof}
		 If $s_{\varepsilon_i+\varepsilon_j,rp} \cdot_{\delta} \lambda = \mu$ then by Proposition $\ref{prop:typed}$ we have either $\mu\subseteq\lambda$ or $\lambda\subseteq\mu$. Since $s_{\varepsilon_i+\varepsilon_j,rp}$ is a reflection we may swap $\lambda$ and $\mu$ if necessary and always assume the former case.
		 
		 Next, notice that
				\begin{eqnarray*}
					s_{\varepsilon_i+\varepsilon_j,rp}\cdot_\delta\lambda &=& s_{\varepsilon_i+\varepsilon_j,rp}(\lambda+\rho(\delta))-\rho(\delta)\\
						&=& \lambda+\rho(\delta)-((\lambda+\rho(\delta),\varepsilon_i+\varepsilon_j)-rp)(\varepsilon_i+\varepsilon_j)-\rho(\delta)\\
						&=& \lambda+\rho(\delta)-
						\left(\lambda_i+\lambda_j-\frac{\delta}{2}-(i-1)-\frac{\delta}{2}-(j-1)-rp\right)(\varepsilon_i+\varepsilon_j)-\rho(\delta)\\
						&=& \lambda+\rho(\delta+rp)-
						\left(\lambda_i+\lambda_j-\frac{\delta+rp}{2}-(i-1)-\frac{\delta+rp}{2}-(j-1)\right)(\varepsilon_i+\varepsilon_j)-
						\rho(\delta+rp)\\
						&=& \lambda+\rho(\delta+rp)-(\lambda+\rho(\delta+rp),\varepsilon_i+\varepsilon_j)(\varepsilon_i+\varepsilon_j)-\rho(\delta+rp)\\
						&=& s_{\varepsilon_i+\varepsilon_j}\cdot_{\delta+rp}\lambda
				\end{eqnarray*}
		Therefore we have $s_{\varepsilon_i+\varepsilon_j}\cdot_{\delta+rp}\lambda=\mu$ with $\mu\subseteq\lambda$, so by Proposition \ref{prop:maximal} we see that $\lambda^T$ and $\mu^T$ are maximal $(\delta+rp)$-balanced. Theorem \ref{thm:maxbal} then shows that
				\[\mathrm{Hom}_{B_n^K(\delta+rp)}(\Delta_n^K(\lambda^T),\Delta_n^K(\mu^T))\neq0\]
		As the cell modules have a basis over $R$ (see Section 2), we can consider the $p$-modular reductions of these and using Lemma \ref{lem:modred} above conclude that,
			\[ \mathrm{Hom}_{B_n^k{(\bar\delta)}}(\Delta_n^k(\lambda^T),\Delta_n^k(\mu^T))\neq0 \]
			
		If now we assume that $\lambda^T$ is $p$-regular, then the simple module $L_n^k(\lambda^T)$ exists and we have a non-zero decomposition number
		\[[\Delta_n^k(\mu^T):L_n^k(\lambda^T)]\neq0\]
	\end{proof}

To prove the corresponding result for reflections of type $s_{\varepsilon_i-\varepsilon_j,rp}$, we will require the following theorem of Carter \& Payne \cite{carterpayne} and an analogue of Proposition $\ref{prop:typed}$.

	\begin{thm}[\cite{carterpayne}]
		Let $\lambda,\mu$ be partitions of n and suppose the Young diagram of $\lambda$ is obtained from that of $\mu$ by raising $d$ nodes from row $j$ to row $i$. If we move the nodes one space at a time, up and to the right, then each will move $\lambda_i-\lambda_j+j-i-d$ spaces. Suppose that this number is divisible by $p^e$, and also that $d<p^e$.\\Then $\mathrm{Hom}_{kS_n}(S^\lambda_k,S^\mu_k)\neq0$
		\label{thm:carterpayne}
	\end{thm}
	
	\begin{propn}
		Let $\lambda,\mu\in\Lambda_n$. If there is a reflection $s_{\varepsilon_i-\varepsilon_j,rp} \in W_p$ $(1\leq i<j\leq n)$ such that
			\[s_{\varepsilon_i-\varepsilon_j,rp} \cdot_{\delta} \lambda = \mu\]
		then $\mu=\lambda-(\lambda_i-\lambda_j-i+j-rp)(\varepsilon_i-\varepsilon_j)$. In particular, $|\lambda|=|\mu|$ and either $\mu\trianglelefteq\lambda$ or $\mu\trianglelefteq\lambda$, where $\trianglelefteq$ is the dominance order on partitions.
		\label{prop:typea}
	\end{propn}
		
		\begin{proof}
			We have	
			\begin{eqnarray*}
				s_{\varepsilon_i-\varepsilon_j,rp}\cdot_\delta\lambda &=& s_{\varepsilon_i-\varepsilon_j,rp}(\lambda+\rho(\delta))-\rho(\delta)\\
					&=& \lambda+\rho(\delta)-((\lambda+\rho(\delta),\varepsilon_i-\varepsilon_j)-rp)(\varepsilon_i-\varepsilon_j)-\rho(\delta)\\
					&=& \lambda+\rho(\delta)-
						\left(\lambda_i-\lambda_j-\frac{\delta}{2}-(i-1)+\frac{\delta}{2}+(j-1)-rp\right)(\varepsilon_i-\varepsilon_j)-\rho(\delta)\\
					&=& \lambda-\left(\lambda_i-\lambda_j-i+j-rp\right)(\varepsilon_i-\varepsilon_j)
			\end{eqnarray*}
		So the effect of $s_{\varepsilon_i-\varepsilon_j,rp}$ is to remove boxes from one row of the Young diagram of $\lambda$ and add the same number to another row. It is then clear that $|\lambda|=|\mu|$, and it remains to consider the following three cases:
		\begin{itemize}
		\item If $(\lambda_i-\lambda_j-i+j-rp)=0$, then $\lambda=\mu$ and the result follows trivially.
		\item If $(\lambda_i-\lambda_j-i+j-rp)>0$, then we are moving nodes in the Young diagram of $\lambda$ from row $i$ into row $j$. Since $i<j$, we are moving nodes into a lower row and therefore $\mu\trianglelefteq\lambda$.
		\item If $(\lambda_i-\lambda_j-i+j-rp)<0$ then we move the nodes from row $j$ to row $i$. Since $j>i$, we are moving nodes to an earlier row and therefore $\lambda\trianglelefteq\mu$.
		\end{itemize}
	\end{proof}

We can now prove the following:

	\begin{thm}
		Let $\lambda,\mu\in\Lambda_n$, with $\lambda^T$ $p$-regular. If there is a reflection $s_{\varepsilon_i-\varepsilon_j,rp} \in W_p$ such that
			\[s_{\varepsilon_i-\varepsilon_j,rp} \cdot_{\delta} \lambda = \mu\]
		then without loss of generality $\mu\trianglelefteq\lambda$ and $\mathrm{Hom}_{B_n^k(\bar\delta)}(\Delta_n^k(\lambda),\Delta_n^k(\mu))\neq0$. In particular, we have a non-zero decomposition number  $[\Delta_n^k(\mu^T):L_n^k(\lambda^T)]\neq0$
	\end{thm}
	
		\begin{proof}
		If $s_{\varepsilon_i-\varepsilon_j,rp} \cdot_{\delta} \lambda = \mu$ then by Proposition $\ref{prop:typea}$ we have either $\mu\trianglelefteq\lambda$ or $\lambda\trianglelefteq\mu$. Since $s_{\varepsilon_i-\varepsilon_j,rp}$ is a reflection we may swap $\lambda$ and $\mu$ if necessary and always assume the former case, i.e. that we are raising nodes in the Young diagram. We can also assume $r\neq0$, since otherwise $s_{\varepsilon_i-\varepsilon_j,rp} \cdot_{\delta} \lambda$ cannot be a partition.\\
			\\
			From Proposition \ref{prop:typea} we set $d=\left(\lambda_i-\lambda_j-i+j-rp\right)$ and $e=1$ in the context of Theorem \ref{thm:carterpayne}, and obtain
				\[\lambda_i-\lambda_j-i+j-d=rp\]
			which is divisible by $p$ as $r\neq0$. The condition $d<p^e$ is also satisfied, since if we were able to move $p$ or more nodes then $\lambda^T$ would not be $p$-regular. Therefore we may apply Theorem \ref{thm:carterpayne} to deduce
				\begin{equation}
					\mathrm{Hom}_{kS_m}(S^\lambda_k,S^\mu_k)\neq0 \label{eqn:cphom}
				\end{equation}
			where $m=|\lambda|$, and thus by Theorem \ref{thm:symblocks},
				\[ \mathrm{Hom}_{B_n^k{(\bar\delta)}}(\Delta_n^k(\lambda),\Delta_n^k(\mu))\neq0 \]
			For the final result, recall from \cite[Section 7]{jameskerber} that $S_k^\lambda\otimes S_k^{(1^m)}\cong S_k^{\lambda^T}$. Using this and \eqref{eqn:cphom}, we have
			\[\mathrm{Hom}_{kS_m}(S^{\lambda^T}_k,S^{\mu^T}_k)\neq0\]
			and hence
			\[ \mathrm{Hom}_{B_n^k{(\bar\delta)}}(\Delta_n^k(\lambda^T),\Delta_n^k(\mu^T))\neq0 \]
			by Theorem \ref{thm:symblocks}. Now as $\lambda^T$ is $p$-regular, the simple module $L_n^k(\lambda^T)$ exists and we must then have a non-zero decompostion number
			\[[\Delta_n^k(\mu^T):L_n^k(\lambda^T)]\neq0\]
		\end{proof}


\section*{Acknowledgements}
The author would like to thank Maud De Visscher and Joseph Chuang for their comments, and for many interesting and helpful discussions.


~\\
\textsc{Centre for Mathematical Science, City University London, Northampton Square, London, EC1V 0HB, United Kingdom}\\
\emph{E-mail address:} \verb+oliver.king.1@city.ac.uk+


\begin{thebibliography}{9}

\bibitem{brauer}
  R. Brauer,
  \emph{On algebras which are connected with the semisimple continuous groups},
  Ann. of Math. (1937),
  \textbf{38},
  857-872.

\bibitem{grahamlehrer}
  J. J. Graham and G. I. Lehrer,
  \emph{Cellular algebras},
  Invent. Math. (1996),
  \textbf{123},
  1-34.
  
\bibitem{wenzl}
  H. Wenzl,
  \emph{On the structure of Brauer's centralizer algebras},
  Ann. of Math. (1988),
  \textbf{128},
  173–193.

\bibitem{rui}
  H. Rui,
  \emph{A criterion on the semisimple Brauer algebras},
  J. Comb. Theory Ser. A (2005),
  \textbf{111},
  78-88.  
  
\bibitem{blocks}
  A. Cox, M. De Visscher and P. Martin,
  \emph{The blocks of the Brauer algebra in characteristic zero},
  Representation Theory (2009),
  \textbf{13},
  272-308.
  
\bibitem{cmpx}
  A. Cox, P. Martin, A. Parker and C. Xi,
  \emph{Representation theory of towers of recollement: theory, notes and examples},
  J. Algebra (2006),
  \textbf{302},
  340-360.
  
\bibitem{geom}
  A. Cox, M. De Visscher and P. Martin,
  \emph{A geometric characterisation of the blocks of the Brauer algebra},
  J. London Math. Soc. (2009),
  \textbf{80},
  471-494.

\bibitem{jameskerber}
  G. D. James and A. Kerber,
  \emph{The Representation Theory of the Symmetric Groups},
  Encyclopedia of Mathematics and its Applications, vol. 16,
  Addison-Wesley,
  1981.

\bibitem{green}
  J. A. Green,
  \emph{Polynomial Representations of $GL_n$},
  Lecture Notes in Mathematics 830,
  Springer,
  1980.
   
\bibitem{hartmannpaget}
  R. Hartmann and R. Paget,
  \emph{Young Modules and filtration multiplicities for Brauer algebras},
  Math.Z. (2006),
  \textbf{254},
  333-357.
  
\bibitem{carterpayne}
  R. W. Carter and M. T. J. Payne,
  \emph{On homomorphisms between Weyl Modules and Specht modules}.
  Math. Proc. Camb. Phil. Soc. (1980),
  \textbf{87},
  419.

\end{thebibliography}
\end{document}